\documentclass[11pt,epsf]{article}
\usepackage[dvips]{color}
\usepackage{geometry}
\newcommand{\bed}{\begin{displaymath}}
\newcommand{\eed}{\end{displaymath}}
\newcommand{\bea}{\bed\begin{array}{rl}}
\newcommand{\eea}{\end{array}\eed}
\newcommand{\barray}{\begin{array}{ll}}
\newcommand{\earray}{\end{array}}

\usepackage{color,latexsym,amsfonts,amssymb,amsmath,amsthm}
\usepackage{colordvi,amsfonts,amsmath,epsfig,subfigure,cite}
\usepackage{graphicx}
\usepackage{caption}
\newtheorem{theorem}{Theorem}[section]
\newtheorem{lemma}[theorem]{Lemma}

\newtheorem{corollary}[theorem]{Corollary}

\newtheorem{definition}{Definition}[section]

\usepackage{fancyhdr}
\begin{document}

\title{ Stability in terms of two measures of solutions to stochastic partial differential delay equations with switching  }
\author{Shufen Zhao$^{a,b}$\thanks{Corresponding author. zsfzx1982@sina.com, 12b312003@hit.edu.cn. This work is supported by the NSF of P.R. China (No.11071050)}
,~~Minghui Song$^{a}$
\\$a$ Department of Mathematics, Harbin Institute of Technology, Harbin 150001, PR China\\
$b$ Department of  Mathematics, Zhaotong University, Zhaotong 657000, PR China}
\maketitle
\begin{abstract}
In this paper, the problem of stability in terms of two measures is considered for a class of stochastic partial differential delay equations with switching. Sufficient conditions for stability in terms of two measures are obtained based on the technique of constructing a proper approximating strong solution system and carrying out a limiting type of argument to pass on stability of strong solutions to mild ones obtained by Bao, Truman and Yuan [ J. Bao, A. Truman,C. Yuan, Stability in distribution of mild solutions to stochastic partial differential delay equations with jumps, Proc. R. Soc. A, 465, 2111-2134 (2009)]. In particular, the stochastic stability under the fixed-index sequence monotonicity condition and under the average dwell-time switching are considered.

{\bf Keywords}
stochastic partial differential equations;  stability in terms of two measures ;

 average dwell-time switching \\
{\bf Mathematics Subject Classfication} 93E03 ;  93E15.

\end{abstract}
\pagestyle{fancy}
\fancyhf{}
\fancyhead[CO]{ Stability in terms of two measures of solutions to stochastic partial differential delay equations with switching }
\fancyhead[LE,RO]{\thepage}
\section{Introduction}
Over past decades, stability theory for solutions of stochastic differential equations have attracted more and more attention due to its welled formulation and analysis in mechanical, electrical, control engineering and physical sciences (see,~e.g.,~\cite{01,02,03,04,gn,4,luj,I,song} and the references therein ). Especially, several authors have obtained conditions for stability in terms of two measures of some stochastic differential equations, such as Chatterjee and Liberzon in \cite{s} presented a general framework for analyzing stability of deterministic and stochastic switched systems via a comparison principle and multiple Lyapunov functions. Yuan \cite{shigui} investigated the stability in terms of two measures for stochastic differential equations with Markovian switching by using the method of Lyapunov functions. Yao and Deng in \cite{y1,y2} discussed the stability in terms of two measures for impulsive stochastic functional differential systems via comparison approach. For the deterministic systems, there have been many researches about the stability in terms of two measures, Liu and Wang \cite{jia1} considered the stability in terms of two measures for impulsive systems of functional differential equations, Ahmad \cite{jia2} obtained sufficient conditions for the stability in terms of two measures for perturbed impulsive delay integro-differential equations, Aleksandr and Slyn'ko \cite{jia3} discussed the stability in terms of two measures for a class of semilinear impulsive parabolic equations. However, questions about the stability of solutions for stochastic partial differential equations are less well understood.

In general, the existing results about the solution for the stochastic system are necessary, and some of the previous work were under the assumption that
there have at least one solution for the system (see,~e.g.,~\cite{y1,y2}). Actually, there are many nonlinear stochastic partial differential equations which do not satisfy the assumption evidently. In \cite{1}, Bao, Truman and Yuan considered some sufficient conditions for the stability in distribution of mild solution to stochastic partial delay equations with jumps by the technique developed in \cite{1} which based on the technique of constructing a proper approximating strong solution system and carrying out a limiting type of argument to pass on stability of strong solutions to mild ones. A family of continuous-time systems, together with a switching signal that choose an active subsystem from the family at every instant of time constitute a switched system \cite{sw}, which have been observed in many practical systems \cite{gn}.
Motivated by the above discussion, our aim is to consider the stability in terms of two measures for the solution of a class of stochastic partial delay equations with switching. We get sufficient conditions for stability in terms of two measures. The results improve and generalize those in earlier publications.

 We first introduce some preliminaries such as the conditions that sufficient for the existence of the solution in Section 2.  Section 3 devotes
to the main stability results, followed by the results under fixed-index sequence monotonicity condition in Section 4.
Stochastic stability under dwell-time switching is considered in Section 5. At last, we give some concluding remarks in Section 6.
\section{Preliminaries}
Let $(\Omega,\mathcal{F},\mathbf{P})$ be a complete probability space equipped with some filtration $\{\mathcal{F}_{t}\}_{t\geq0}$ satisfying the usual conditions , i.e., the filtration is right continuous and $\mathcal{F}_{0}$ contains all $\mathbf{P}$-null sets. Let $H$, $K$ be two real separable Hilbert spaces and denote by $<\cdot,\cdot>_{H}$,  $<\cdot,\cdot>_{K}$ their inner products and by $\|\cdot\|_{H}$, $\|\cdot\|_{K}$ their vector norms, respectively. Let $\mathcal{L}(K,H)$ be the set of all inner bounded operators from $K$ into $H$, equipped with the usual operator norm $\|\cdot\|$. Let $\tau>0$ and $D([-\tau,0];H)$ denote the family of all right continuous functions with left hand limits $\varphi$ from $[-\tau,0]$ to $H$ equipped with the norm $\|\varphi\|_{D}:=\sup_{-\tau\leq\theta\leq 0}\|\varphi\|_{H}$. We use $D^{b}_{\mathcal{F}_{0}}([-\tau,0];H)$ to denote the family of all almost surely bounded, $\mathcal{F}_{0}$ measurable, $D([-\tau,0];H)$ valued random variables.

Let $\{W(t),t\geq 0\}$ denote a $K$-valued $\{\mathcal{F}_{t}\}_{t\geq0}$ adapted Wiener process defined on the probability space $(\Omega,\mathcal{F},\mathbf{P})$ with covariance operator $Q$, i.e.
\begin{equation*}
    \mathbf{E}<W(t),x>_{K}<W(t),y>_{K}=(t\wedge s)<Qx,y>_{K},\:\forall x,y\in K,
\end{equation*}
where $Q$ is a positive, self-adjoint, trace class operator on $K$. In particular we call such $\{W(t),t\geq 0\}$ a $K$ valued $Q$ wiener process relative to $\{\mathcal{F}_{t}\}_{t\geq0}$ just as presented in \cite{1}, $W(t)$ is defined by
 $W(t)=\sum_{n=1}^{\infty}\sqrt{\lambda_{n}}\beta_{n}(t)e_{n},\,t\geq0,$
where $\{\beta_{n}(t)\}_{n\in \mathbf{N}}$ is a sequence of real valued standard Brownian motions mutually independent on the probability space $(\Omega,\mathcal{F},\mathbf{P})$, $\lambda_{n},~n\in \mathbf{N}$ are the eigenvalues of $Q$ and $e_{n},~n\in \mathbf{N},$ are the corresponding eigenvectors. That is
\begin{equation*}
    Qe_{n}=\lambda_{n}e_{n},\,n=1,2,3,\ldots.
\end{equation*}
We introduce the subspace $K_{0}=Q^{1/2}(K)$ of $K$, which endowed with the inner product
   $ <u,v>_{K_{0}}=<Q^{1/2}u,Q^{1/2}v>_{K}$
is a Hilbert space. Let $\mathcal{L}_{2}^{0}=\mathcal{L}_{2}(K_{0},H)$ denote the space of all Hilbert Schmidt operators from $K_{0}$ into $H$. It turns to be a separable Hilbert space, equipped with the norm
\begin{equation*}
\|\psi\|^{2}_{\mathcal{L}_{2}^{0}}=tr\big((\psi Q^{1/2})(\psi Q^{1/2})^{*}\big)\;\textrm{for any }\;\psi\in \mathcal{L}_{2}^{0}.
\end{equation*}
Clearly, for any bounded operators $\psi\in \mathcal{L}(K,H)$, this norm reduces to $\|\psi\|^{2}_{\mathcal{L}_{2}^{0}}=tr(\psi Q\psi^{*})$.
Let $\Phi:(0,\infty)\rightarrow\mathcal{L}_{2}^{0}$ be predictable, $\mathcal{F}_{t}$ adapted process such that
\begin{equation*}
\int_{0}^{t}\mathbf{E}\|\Phi(s)\|^{2}_{\mathcal{L}_{2}^{0}}\mathrm{d}s<\infty,\, \forall t>0.
\end{equation*}
Then we can define the $H$ valued stochastic integral
\begin{equation*}
\int_{0}^{t}\Phi(s)\mathrm{d}W(s),
\end{equation*}
which is a continuous square-integrable martingale \cite{2}.
Let $P=P(t)$, $t\in D_{P}$ be a stationary $\mathcal{F}_{t}$ Poisson point process with characteristic measure $\lambda$ (see,\cite{situ,3}). Denote by $N(\mathrm{d}t,\mathrm{d}u)$ the Poisson counting measure associated with $P$, i.e.
\begin{equation*}
N(t,\mathbf{Z})=\sum_{s\in D_{P},s\leq t}I_{\mathbf{Z}}(P(s))
\end{equation*}
 with measurable set $\mathbf{Z}\in \mathcal{B}(K-\{0\})$, which denotes the Borel $\sigma$ field of $K-\{0\}$. Let $\tilde{N}(\mathrm{d}t,\mathrm{d}u):=N(\mathrm{d}t,\mathrm{d}u)-\mathrm{d}t\lambda(\mathrm{d}u)$ be the compensated  Poisson measure that is independent of $W(t)$.
In the following, we investigate the stability in terms of two measures of the following stochastic partial differential equations with jumps and switching in the following form: let $\mathcal{S}$ is an index set, for given $\tau>0$ and arbitrary $t\geq 0,$
\begin{eqnarray}\label{eq1}
\mathrm{d}X(t)&=&[AX(t)+F_{\sigma}(X(t),X(t-\tau))\mathrm{d}t+G_{\sigma}(X(t),X(t-\tau))\mathrm{d}W(t)\\
\nonumber&&+\int_{\mathbf{Z}}L_{\sigma}(X(t),X(t-\tau),u)\tilde{N}(\mathrm{d}t,\mathrm{d}u),
\end{eqnarray}
with initial datum $X(t)=\xi(t)\in D^{b}_{\mathcal{F}_{0}}([-\tau,0];H)$, $-\tau\leq t\leq 0,$ $\sigma :\mathbf{R}_{+}\rightarrow \mathcal{S}$ is a piecewise constant function, which specifies at every time $t,$ the index $\sigma(t)=p\in \mathcal{S}$ and $\sigma(0)=p_{0}.$ As in \cite{1}, the following assumptions are imposed for the existence and uniqueness of the mild solution to (\ref{eq1}).
\begin{enumerate}
  \item [(H1)] $A$, generally unbounded, is the infinitesimal generator of a $C_{0}$ semigroup $T(t)$, $t\geq 0$, of contraction.
  \item [(H2)] The mapping $F_{p}:H\times H\rightarrow H$, $G_{p}:H\times H\rightarrow\mathcal{L}_{2}^{0}$ and $L_{p}:H\times H\times \mathbf{Z}\rightarrow H$ ( $\forall\,p\in \mathcal{S}$) are Borel measurable and satisfy the following Lipschitz continuity condition and linear growth condition for some constant $k>0$ and arbitrary $x,y,x_{1}, x_{2},y_{1}, y_{2}\in H$,
     \begin{eqnarray*}
      && \|F_{p}(x_{1},y_{1})-F_{p}(x_{2},y_{2})\|_{H}^{2}+\|G_{p}(x_{1},y_{1})-G_{p}(x_{2},y_{2})\|_{\mathcal{L}_{2}^{0}}\\
       &&\quad+\int_{\mathbf{Z}}\|L_{p}(x_{1},y_{1},u)-L_{p}(x_{2},y_{2},u)\|_{H}^{2}\lambda(\mathrm{d}u)\\
       &&\leq k(\|x_{1}-x_{2}\|_{H}^{2}+\|y_{1}-y_{2}\|_{H}^{2})
     \end{eqnarray*}
      and  \begin{eqnarray*}
         &&\|F_{p}(x,y)\|_{H}^{2}+ \|G_{p}(x,y)\|_{\mathcal{L}_{2}^{0}}^{2}+\int_{\mathbf{Z}}\|L_{p}(x,y,u)\|_{H}^{2}\lambda(\mathrm{d}u)\\
        && \leq k(1+\|x\|_{H}^{2}+\|y\|_{H}^{2}).
           \end{eqnarray*}
    \item [(H3)] There exists a number $L_{0}>0$ such that for arbitrary $x,y,x_{1}, x_{2},y_{1}, y_{2}\in H,$
    \begin{eqnarray*}
    &&\int_{\mathbf{Z}}\|L_{p}(x_{1},y_{1},u)-L_{p}(x_{2},y_{2},u)\|_{H}^{4}\lambda(\mathrm{d}u)\leq L_{0}(\|x_{1}-x_{2}\|_{H}^{4}+\|y_{1}-y_{2}\|_{H}^{4}),
    \end{eqnarray*}
    \begin{eqnarray*}&&\int_{\mathbf{Z}}\|L_{p}(x,y,u)\|_{H}^{4}\lambda(\mathrm{d}u)\leq L_{0}(1+\|x\|_{H}^{4}+\|y\|_{H}^{4}).\end{eqnarray*}
\end{enumerate}
For all $t\geq0,$ $X_{t}=\{X(t+\theta):-\tau\leq\theta\leq0\}$ is regarded as a $D([-\tau,0];H)$-valued stochastic process.
\begin{definition}
An $H$-valued stochastic process $\{X(t),t\in[-\tau,$ $T]\},$ $0\leq T<\infty$ is called a strong solution of equation (\ref{eq1}) if
\begin{enumerate}
  \item [(i)] For any $t\in[0,T],$ $X_{t}\in D([-\tau,0];H)$ is adapted to $\mathcal{F}_{t}.$
  \item [(ii)] $X(t)\in H$ has c\`{a}dl\`{a}g path on $t\in[0,T]$ almost surely, $X(t)\in D(A)$ on $[0,T]\times\Omega$ with $\int_{0}^{T}\|AX(t)\|_{H}\mathrm{d}t<\infty$ almost surely and for all $t\in [0,T],$
 \begin{eqnarray*}
 X(t)&=&\xi(0)+\int_{0}^{t}[AX(s)+F_{\sigma(s)}(X(s),X(s-\tau))\mathrm{d}s\\&&+\int_{0}^{t}G_{\sigma(s)}(X(s),X(s-\tau))\mathrm{d}W(s)\\
 &&+\int_{0}^{t}\int_{\mathbf{Z}}L_{\sigma(s)}(X(s),X(s-\tau),u)\tilde{N}(\mathrm{d}s,\mathrm{d}u),
  \end{eqnarray*}
  for any $X_{0}(\cdot)=\xi(\cdot)\in D_{\mathcal{F}_{0}}^{b}([-\tau,0];H)$, $-\tau\leq t\leq0$.
\end{enumerate}
\end{definition}
\begin{definition}
An $H$-valued stochastic process $\{X(t),t\in [-\tau,$ $T]\}$, $0\leq T<\infty$ is called a mild solution of equation (\ref{eq1}) if
\begin{enumerate}
  \item [(i)]  For any $t\in[0,T],$ $X_{t}(\cdot)\in D([-\tau,0];H)$ is adapted to $\mathcal{F}_{t}.$
  \item [(ii)] $X(t)\in H$ has c\`{a}dl\`{a}g path on $t\geq 0$ almost surely, and for arbitrary $t\in[0,T]$,
  \begin{eqnarray*}
 X(t)&=&T(t)\xi(0)+\int_{0}^{t}T(t-s)F_{\sigma(s)}(X(s),X(s-\tau))\mathrm{d}s\\&&+\int_{0}^{t}T(t-s)G_{\sigma(s)}(X(s),X(s-\tau))\mathrm{d}W(s)\\
 &&+\int_{0}^{t}\int_{\mathbf{Z}}T(t-s)L_{\sigma(s)}(X(s),X(s-\tau),u)\tilde{N}(\mathrm{d}s,\mathrm{d}u),
  \end{eqnarray*}
  for any $X_{0}(\cdot)=\xi(\cdot)\in D_{\mathcal{F}_{0}}^{b}([-\tau,0];H)$, $-\tau\leq t\leq0$.
\end{enumerate}
\end{definition}
 The following two conclusions appeared in \cite{1} which given theoretical basis for the stability analysis.
 \begin{lemma}(\cite{1})
 Under the assumptions (H1), (H2) and (H3), system (\ref{eq1}) admits a unique mild solution.
 \end{lemma}
\begin{lemma}(\cite{1})
  Let $\xi\in D^{b}_{\mathcal{F}_{0}}([-\tau,0];H)$ be an arbitrarily given initial datum and assume that conditions (H1), (H2) and (H3) hold, then (\ref{eq22}) has a unique strong solution $X^{n}(t)\in D(A)$, which lies in $C([0,T];L^{2}(\Omega,\mathcal{F},\mathbf{P};H))$ for all $T>0$. Moreover, $X^{n}(t)$ converges to the mild solution $X(t)$ of system (\ref{eq1}) almost surely in $C([0,T];L^{2}(\Omega,\mathcal{F},\mathbf{P};H))$ as $n\rightarrow \infty.$
 \end{lemma}
 \begin{equation}\label{eq22}
    \left\{ \begin{array}{ll}
 \mathrm{d}X(t)&=AX(t)\mathrm{d}t+R(n)F_{\sigma}(X(t),X(t-\tau))\mathrm{d}t+R(n)G_{\sigma}(X(t),X(t-\tau))\mathrm{d}W(t)\\&\quad+\int_{\mathbf{Z}}R(n)L_{\sigma}(X(t),X(t-\tau),u)\tilde{N}(\mathrm{d}t,\mathrm{d}u), \\
X(t)&=R_{\sigma}(n)\xi(t)\in D(A),\quad -\tau\leq t\leq0,
\\\sigma(0)&=p_{0},
\end{array} \right.
\end{equation}
where $n\in \rho(A)$, the resolvent set of $A$ and $R(n)=nR(n,A)$, $R(n,A)$ is the resolvent of $A.$
Let $C^{2}(H\times\mathcal{S};\mathbf{R}_{+})$ denote the space of all real valued nonnegative functions $V$ with properties:
\begin{enumerate}
  \item[(i)] for any $p\in\mathcal{S},$ $V(x,p)$ is twice (Fr\'{e}chet) differentiable in $x,$
  \item[(ii)]for any $p\in\mathcal{S},$ $V_{x}(\cdot,p)$ and $V_{xx}(\cdot,p)$ are both continuous in $H$ and $\mathcal{L}(H)=\mathcal{L}(H,H)$, respectively.
\end{enumerate}
Suppose $V\in C^{2}(H\times\mathcal{S};\mathbf{R}_{+})$, let $X(t)$ be a strong solution of equation (\ref{eq1}), then with $t\geq 0$, according to the work in \cite{4,1}, we present the following It\^{o} formula.
\begin{eqnarray*}\small
V(X(t),\sigma)&=&V(\xi,p_{0})+\int_{0}^{t}\mathcal{L}V(X(s),X(s-\tau),\sigma)\mathrm{d}s\\&&~+\int_{0}^{t}<V_{x}(X(s)),G_{\sigma}(X(s),X(s-\tau))\mathrm{d}W(s)>_{H}\\
&&~+\int_{0}^{t}\int_{\mathbf{Z}}[V(X(s)+L_{\sigma}(X(s),X(s-\tau),u))-V(X(s))]\tilde{N}(\mathrm{d}s,\mathrm{d}u).
\end{eqnarray*}
For any $(\varphi,p)\in D([-\tau,0];H)\times\mathcal{S}$ with $\varphi(0)\in D(A)$ where $D(A)$ is the domain of operator $A$, we introduce the following:
\begin{eqnarray*}\small
&&\mathcal{L}V(\varphi,p)\\&&=<V_{x}, A\varphi(0)+F_{p}(\varphi(0),\varphi)>_{H}+\frac{1}{2}tr(V_{xx}(x)G_{p}(\varphi(0),\varphi)QG^{*}_{\sigma}(\varphi(0),\varphi))
\\&&\quad+\int_{\mathbf{Z}}[V(\varphi(0)+L_{\sigma}(\varphi(0),\varphi,u))-V(\varphi(0),p)
\\&&\quad\quad\quad
-<V_{x}(\varphi(0),p),L_{p}(\varphi(0),\varphi,u)>_{H}]\lambda(\mathrm{d}u).
\end{eqnarray*}
We say a function $\alpha\in C(\mathbf{R}_{+},\mathbf{R}_{+})$ is of class $\mathcal{K}$ if $\alpha$ is strictly increasing with $\alpha(0)=0$, is of class $\mathcal{K}_{\infty}$ if in addition $\alpha(r)\rightarrow\infty$ as $r\rightarrow\infty$; and we write $\alpha\in\mathcal{K}$ and $\alpha\in\mathcal{K}_{\infty}$ respectively. A function $\beta\in C(\mathbf{R}_{+}^{2},\mathbf{R}_{+})$ is said to be of class $\mathcal{KL}$ if $\beta(\cdot,t)$ is a function of class $\mathcal{K}$ for every fixed $t$ and $\beta(r,t)\rightarrow 0$ as $t\rightarrow\infty$ for every fixed $r$; and we write $\beta\in\mathcal{KL}.$ Let
\begin{equation*}
 \Gamma:=\{h\in C(H,\mathbf{R}_{+})|\inf_{x\in H}h(x)=0\}.
\end{equation*}
\begin{definition}\label{def1}
Let $h^{\circ}$, $h\in \Gamma$. System (\ref{eq1}) is said to be $(h^{\circ},h)$ globally asymptotically stable in the mean ($(h^{\circ},h)$-GAS-M) if there exists a function $\beta\in\mathcal{KL}$ such that for every $\xi(t)\in D^{b}_{\mathcal{F}_{0}}([-\tau,0];H)$, the inequality
\begin{equation}\label{ieq2}
    \mathbf{E}[h(X_{t})]\leq \beta(h^{\circ}(\xi),t),\quad \forall t\geq0,
\end{equation}
holds.
\end{definition}
\begin{definition}\label{def2}System (\ref{eq1}) is said to be $(h^{\circ},h)$ globally  asymptotically stable in probability ($(h^{\circ},h)$-GAS-P) if for every $\eta\in[0,1],$ there exists a function $\beta\in\mathcal{KL}$ such that for every $\xi(t)\in D^{b}_{\mathcal{F}_{0}}([-\tau,0];H)$, the inequality
\begin{equation}\label{eqq2}
  \mathbf{P}[h(X_{t})\geq \beta(h^{\circ}(\xi),t)]<\eta,\quad \forall t\geq0,
\end{equation}
holds.
\end{definition}
\section{Main results }
In the previous section, we showed the relation between the strong solution of (\ref{eq22}) and the mild solution of (\ref{eq1}). However, in order to get our main results about the two measures stability of (\ref{eq1}), we also need the following lemma.

 Let $V\in C^{2}(H;\mathbf{R}_{+})$. The function $V$ is said to be $h$-positive definite if there exists a function $\alpha_{1}\in \mathcal{K}_{\infty}$ such that for any $\varphi\in D,$ $\alpha_{1}\circ h(\varphi)\leq V(\varphi(0))$, and $h^{\circ}$-decrescent if there exists a function $\alpha_{2}\in \mathcal{K}_{\infty}$ such that $V(\varphi(0))\leq \alpha_{2}\circ h^{\circ}(\varphi).$
\begin{lemma}\label{lem1}(\cite{8})
Suppose $\phi(u,\psi):\mathbf{R}_{+}\times D([-\tau,0];\mathbf{R}_{+})\rightarrow\mathbf{R} $ is a continuous function which is non-decreasing with respect to $\psi\in D([-\tau,0];\mathbf{R}_{+})$, i.e. if $\forall$ $\psi_{1},\psi_{2}\in D([-\tau,0];\mathbf{R}_{+})$ with $\psi_{1}(\theta)\leq\psi_{2}(\theta),$ $\theta\in[-\tau,0]$, we have $\phi(u,\psi_{1}(\theta))\leq \phi(u,\psi_{2}(\theta))$ for each $u\in \mathbf{R}_{+}$. Then, for arbitrary given initial function $\psi\in D([-\tau,0];\mathbf{R}_{+})$, there exists some $T$ such that the following equation
 \begin{equation}\label{ceq}
    \Sigma: \left\{ \begin{array}{ll}
 \dot{u}(t)=\phi(u(t),u_{t}),\, t\in [0,T],\\ \qquad u_{0}=\psi,
\end{array} \right.
\end{equation}
admits a unique maximal solution, denote by $\bar{u}(t,\psi)$, defined on $[0, T]$.
\end{lemma}
\begin{lemma}\label{lem2}
Assume the following conditions hold.
\begin{enumerate}
  \item [(A1)] In addition  to the conditions in Lemma \ref{lem1}, the function $\phi(u,\psi)$ is concave in $u\in \mathbf{R}_{+}$ and $\psi\in D([-\tau,0];\mathbf{R}_{+})$.
  \item [(A2)] There exists a function $V(x,p)\in C^{2,0}(H\times \mathcal{S}; \mathbf{R}_{+})$ such that for any $p\in S$ and $\varphi\in D([-\tau,0];H)$ with $\varphi(0)\in D(A)$,
      \begin{equation}\label{conv}
       (\mathcal{L}V)(\varphi,p)\leq \phi(V(\varphi(0),p),V(\varphi,p)(\cdot))
      \end{equation}
       where
       \begin{equation*}
       V(\varphi,p)(\cdot):=\{V(\varphi(\theta),p),\theta\in[-\tau,0]\}\in  D([-\tau,0];\mathbf{R}_{+}).\end{equation*}
  \item [(A3)] The maximal solution $\bar{u}(t,\psi)$ of exists in $[0,\infty)$,i.e.,$ T=\infty$.
  Let $X(t)=X(t,\xi), t\geq0$, denote a mild solution of (\ref{eq1}) with initial $\xi(\cdot)\in D_{\mathcal{F}_{0}}^{b}([-\tau,0];H)$, and if $\mathbf{E}(V(\theta),p)\leq \psi(\theta)$, $\theta\in [-\tau,0]$, $P\in \mathcal{S}$, then
  \begin{equation}\label{ceq1}
    \mathbf{E}V(X(t), p)\leq \bar{u}(t,\psi),\,t\geq0,\,p\in \mathcal{S}
  \end{equation}
\end{enumerate}
\end{lemma}
The proof process is similar to that of Theorem 5.1 in \cite{4}, and we omit it.
\begin{theorem}\label{thm1}
For system (\ref{eq1}) with a sequence of switching instants $\{\tau_{j}\}_{j\geq1}$ and functions $h^{\circ},\,h\in \Gamma$. Suppose that there exists functions $\alpha_{1},\, \alpha_{2}\in \mathcal{K}_{\infty}$, $V(\cdot,p)\in C^{2}(H;\mathbf{R}_{+}) $ for each $p\in \mathcal{S}$, and a system $\Sigma$ of type (\ref{ceq}) such that \begin{enumerate}
       \item [(i)]  The conditions of Lemma \ref{lem2} hold.
       \item [(ii)] 
         $\forall \;(\varphi,p)\in D\times\mathcal{S},$ $\alpha_{1}h(\varphi)\leq V(\varphi(0),p)\leq \alpha_{2}h^{\circ}(\varphi)$ and
         $\alpha_{1}$ is convex.
          \item [(iii)] $\forall \xi\in D,$ there exists $\psi\in D([-\tau,0];\mathbf{R}_{+})$ such that
          \begin{equation*}
            \mathbf{E}[V(X^{n}(t),\sigma(\tau_{i}))]\leq \bar{u}(\tau_{i},\psi)~\text{for~all}~i\geq 0,
          \end{equation*}
       where $X^{n}(t)$ and $\bar{u}(\tau_{i},\psi)$ are the solution of (\ref{eq22}) and the solution of $\Sigma,$ respectively.
       \item [(iv)] $\Sigma$ is globally asymptotic stable in the sense that there is a function $\beta_{\psi}\in\mathcal{KL}$ such that the inequality
       \begin{equation}
        |\bar{u}(t,\psi)|\leq \beta_{\psi}(\|\psi\|_{D},t),\quad\forall\, t\geq0,
       \end{equation}
       holds.
     \end{enumerate}
Then system (\ref{eq1}) is $(h^{\circ},h)$-GAS-M in the sense of Definition \ref{def1}.
\end{theorem}
\begin{proof}
Consider the interval $[\tau_{l},\tau_{l+1}[ $, with $l$ an arbitrary nonnegative integer. According to hypotheses (iv), (iii) and (i), and Lemma \ref{lem2} we have
\begin{equation}
 \mathbf{E}[V(X^{n}(t),\sigma(\tau_{l}))]\leq \bar{u}(t,\psi),\quad \forall t\in [\tau_{l},\tau_{l+1}[.
\end{equation}
 By the the hypothesis (iv), we know $|\bar{u}(t,\psi)|\leq \beta_{\xi}(\|\psi\|_{D},t),$ $\forall t\in [\tau_{l},\tau_{l+1}[$, we get
\begin{equation}
    \mathbf{E}[V(X^{n}(t),\sigma(\tau_{l}))]\leq \beta_{\xi}(\|\psi\|_{D},t),
\end{equation}
where $X_{n}(t)$ is the solution of system (\ref{eq22}). By (i) and (ii), we have
 \begin{equation*}
\alpha_{1}\mathbf{E}h(X^{n}_{t})\leq \mathbf{E}V(X^{n}(t),\sigma(\tau_{l}))\leq \beta_{\xi}(\|\psi\|_{D},t),~ \forall t\in [\tau_{l},\tau_{l+1}[,
 \end{equation*}
which leads to
\begin{equation*}
    \mathbf{E}h(X^{n}_{t})\leq \alpha_{1}^{-1}\beta_{\xi}(\alpha_{2}h^{\circ}(\psi),t).
\end{equation*}
We take $\beta(z,t):=\alpha_{1}^{-1}\beta_{\xi}(\alpha_{2}(z),t)$, obviously $\beta\in \mathcal{KL}$. Letting $n\rightarrow\infty$ , we get
\begin{equation}
    \mathbf{E}h(X_{t})\leq \beta(h^{\circ}(\psi),t),
\end{equation}
so system (\ref{eq1}) is $(h^{\circ},h)$-GAS-M in the sense of Definition \ref{def1}.\end{proof}
\begin{corollary} Suppose that the conditions in Theorem \ref{thm1} hold, then system (\ref{eq1}) is $(h^{\circ},h)$-GAS-P in the sense of Definition \ref{def2}.
\end{corollary}
 \begin{proof}By Theorem \ref{thm1}, we know there exist a function $\beta(z,s)\in \mathcal{KL}$ such that $\mathbf{E}h(X_{t})\leq \beta(h^{\circ}(\psi),t)$. For any $\varepsilon\in[0,1]$, letting $\tilde{\beta}=2\frac{\beta(z,s)}{\varepsilon}$ for all $(z,s)\in \mathbf{R}^{2}_{+}$, by the Chebyshev's inequality, we get
\begin{eqnarray*}
\mathbf{P}[\mathbf{E}h(X_{t})\leq \tilde{\beta}(h^{\circ}(\psi),t)]\leq \frac{\mathbf{E}h(X_{t})}{\tilde{\beta}(h^{\circ}(\psi),t)}\leq\frac{\varepsilon}{2}<\varepsilon.
\end{eqnarray*}
So system (\ref{eq1}) is $(h^{\circ},h)$-GAS-P in the sense of Definition \ref{def2}.\end{proof}
As mentioned in \cite{s}, once we find a suitable comparison system, the stability of a given switched system can be deduced by Theorem \ref{thm1}.
Since switched systems are usually used to model many physical or man made systems, the diverse practical system can be identified by the different  switching signals in some sense. In the following we established the conditions that sufficient for the $(h^{\circ},h)$-GAS-M property of a switched system under a fixed-index sequence monotonicity condition.
\section{Stability under fixed-index sequence monotonicity condition}
 In this section, we consider stochastic switched system (\ref{eq1}) under fixed-index sequence monotonicity condition.
\begin{theorem}\label{the1}
If there exist functions $\alpha,\, \alpha_{1},\, \alpha_{2},\,\rho,\,\rho,\,U\in \mathcal{K}_{\infty}$. $V(\cdot,p)\in C^{2,1}(H;\mathbf{R}_{+})$ for each $p\in \mathcal{S}$, such that
\begin{enumerate}
  \item [(i)] $\alpha_{1},\,\alpha\circ\alpha_{2}^{-1}$ and $U\circ\alpha_{2}^{-1}$ are convex and $\rho$ is concave;
  \item [(ii)] The family $\{V(\cdot,p)|p\in \mathcal{S}\}$ is $\mathcal{S}$-uniformly $h$ positive definite and $h^{\circ}$ decrescent,
  \item [(iii)] $\forall x\in H  $ , $\forall p\in \mathcal{S}$, the estimate
       $(\mathcal{L}V)(\varphi(0),\varphi,p)\leq -\alpha\circ h^{\circ}(\varphi)$
        holds, where $\varphi(0)\in D(A),$
  \item [(iv)] for every pair of switching time $(\tau_{i}, \tau_{j}), $ $i<j$ such that $\sigma(\tau_{i})=\sigma_{\tau_{j}}=p$ and $\sigma(\tau_{k})\neq p$ for $\tau_{i}<\tau_{k}<\tau_{j}$, the inequality
      \begin{equation}
        \mathbf{E}V(X^{n}(\tau_{j}),p)- \mathbf{E}V(X^{n}(\tau_{i}),p)\leq - \mathbf{E}U\circ h^{\circ}(X^{n}_{\tau_{i}})
      \end{equation}
      holds, where $X^{n}(t) $ is the solution process of (\ref{eq22}).
  \item [(v)] $\forall \varphi\in D  $, we have $\alpha_{2}\circ h^{\circ}(\varphi)\leq \rho\circ\alpha_{1}\circ h (\varphi),$\\
      then (\ref{eq1}) is $(h^{\circ},h)$-GAS-M.
      \end{enumerate}
\end{theorem}
\begin{proof}
We define the following impulsive system as a comparison system
\begin{equation}\label{eqc}
    \left\{ \begin{array}{ll}
 \dot{\zeta}&=-\alpha\circ\alpha_{2}^{-1}(\zeta),\quad t\geq0,~t\neq \tau_{i}\\
\zeta(\tau_{i})&=\mathbf{E}V_{\sigma(\tau_{i})}(\|X_{n}(\tau_{i})\|_{H},\sigma(\tau_{i})),\quad i\geq 1,\:\\
\zeta(0)&=\|\xi(0)\|_{H}
\end{array} \right.
\end{equation}
Similar to the proof progress of the Corollary 3.11 in \cite{s}, we can show system (\ref{eqc}) is globally uniformly asymptotically stable.
By Theorem \ref{thm1}, there exists a function $\beta\in\mathcal{KL}$ such that
\begin{equation}
  \mathbf{E}h(X^{n}_{t})\leq \beta\big(h^{\circ}(R(n)\xi(\cdot))\big), \quad\forall t\geq0,
\end{equation}
and letting $n\rightarrow\infty$, we get
\begin{equation}
  \mathbf{E}h(X_{t})\leq \beta(h^{\circ}(\xi(\cdot)), \quad\forall t\geq0,
\end{equation}
which implies that (\ref{eq1}) is $(h^{\circ},h)$-GAS-M.\end{proof}
It is possible to obtain less conservative stability conditions involving more specific policies on the switching signals. In the following section, we consider the stochastic stability for system (\ref{eq1}) under average dwell-time switching.
\section{Stochastic stability under average dwell-time switching }
In this section we investigate conditions on the average dwell-time $\tau_{a}$ of a switching signal such that the solution (\ref{eq1})
is globally asymptotically stable in the $q$th mean. We need the definition of average dwell-time from \cite{9} for the following results.
\begin{definition}\label{adt}(\cite{9}) For a switching signal $\sigma$ and any $t_{2}>t_{1}>t_{0},$ let $N_{\sigma}(t_{1},t_{2})$ be the switching numbers of $\sigma(t)$ over the interval $[t_{1},t_{2}).$ If $N_{\sigma}(t_{1},t_{2})\leq N_{0}+\frac{t_{2}-t_{1}}{\tau_{a}}$ holds for $N_{0}\geq1,\,\tau_{a}>0,$ then $\tau_{a}$ and $N_{0}$ are called the average dwell-time and the chatter bound, respectively.
\end{definition}
\begin{lemma}\label{lem3}(\cite{zhu1})
Assume that $a_{1}$, $a_{2}$ are two constants and $a_{2}>0.$  $u:[t_{0},\infty)\rightarrow \mathbf{R}_{+}$ satisfy the following delay differential inequality
\begin{equation}
  \dot{u}(t)\leq a_{1}u(t)+a_{2}\|u_{t}\|_{D},\quad t\geq t_{0},
\end{equation} where $\|u_{t}\|_{D}=\sup_{-\tau\leq s\leq 0}\|u(t+s)\|.$
If $a_{1}+a_{2}<0,$ then there exists a positive constant $\lambda$ satisfying $\lambda+a_{1}+a_{2}e^{\lambda\tau}<0$ such that
\begin{equation}
  u(t)\leq \|u_{ t_{0}}\|_{D}e^{-\lambda(t-t_{0})},\quad t\geq t_{0}.
\end{equation}
\end{lemma}
\begin{theorem}Suppose the following assumptions hold for the switched system (\ref{eq1}).
\begin{enumerate}
  \item [(A1)]There exist functions $\alpha_{1},\,\alpha_{2}\,\in \mathcal{K}_{\infty},$ $V(\cdot,p)\in C^{2}(H,\mathbf{R})$ for each $p\in\mathcal{S},$ and $\mu>1$ such that
\begin{equation}\label{a}
 \alpha_{1}(\|x\|_{H}^{q})\leq V(x,\sigma(\tau_{l}))\leq \alpha_{1}(\|x\|_{H}^{q})
\end{equation} with $\alpha_{1}$ convex, $q\geq1.$
  \item  [(A2)]\label{f}
  \begin{eqnarray}
\mathcal{L}V(X^{n}_{t},p)\leq \gamma_{1}V(X^{n}(t),p)+\gamma_{2}\sup_{-\tau \leq s\leq0}V(X_{n}(t+s),p),
  \end{eqnarray}
   where $ \gamma_{1}+\gamma_{2}<0,$ $X^{n}(t) $ is the solution process of (\ref{eq22}).
  \item  [(A3)] There is a positive number $\mu>1$ such that
   \begin{equation}\label{g}
 V(x,p_{1})\leq \mu V(x,p_{2}),~\forall x\in H, \forall p_{1},p_{2}\in \mathcal{S}.
 \end{equation}
\end{enumerate}
 Then (\ref{eq1}) is globally asymptotically stable in the $q$th mean for every switching signal $\sigma$ with average dwell-time
 \begin{equation}\label{e}
  \tau_{a}>\tilde{\tau_{a}}=\frac{\ln \mu}{\lambda},
\end{equation}
where $\lambda$ satisfies $\lambda+\gamma_{1}+\gamma_{2}e^{\lambda\tau}<0.$
\end{theorem}
\begin{proof}
Consider an impulsive differential system
\begin{eqnarray}\label{im1}
\left\{ \begin{array}{ll}
\dot{\xi}(t)=\gamma_{1}\xi(t)+\gamma_{2}\xi_{t},\\
\xi(\tau_{i})=\mu\xi(\tau_{i}^{-}),~i\geq1,~t\geq t_{0},\\
\xi(t_{0})=V(X_{0},\sigma(t_{0})),
\end{array} \right.
\end{eqnarray}
where the sequence of switching instants $\{\tau_{j}\}_{j\geq1}$ corresponding to the switching signal $\sigma(t)$ in (\ref{eq1}).
From condition (\ref{a}) it follows that hypothesis (i) of Theorem \ref{thm1}
is satisfied with $h^{\circ}(X_{t})=h(X_{t})=\|X_{t}(0)\|_{H}^{q}.$
Further, from (\ref{f}) and (\ref{g}) together with the initial condition in (\ref{im1}), it follows that hypothesis (iii) of Theorem \ref{thm1} is satisfied.

Let $T>0$ be arbitrary. Consider the evolution of system (\ref{im1}) from $t=t_{0}$ through $t=T.$ Let $N_{\sigma}(T,t_{0})$ switches on this interval, and let $\nu:=N_{\sigma}(T,t),$ where $N_{\sigma}(T,t)$ as defined in Definition \ref{adt}.

Since $\gamma_{1}+\gamma_{2}<0,$ from Lemma \ref{lem3}, there exists at least a positive constant $\lambda$ such that $\lambda+\gamma_{1}+\gamma_{2}e^{\lambda\tau}<0,$ and
\begin{equation}
\xi(\tau_{i+1}^{-})=\xi(\tau)e^{-\lambda(\tau_{i+1}-\tau_{i})}, \quad 0\leq i\leq \nu,
\end{equation}
and $\xi(T)=\xi(\tau_{\nu})e^{-2\lambda(T-\tau_{\nu})}.$ Combining with the reset equation of (\ref{im1}) and iterating over $i,$ it follows that
\begin{equation}\label{xi}
  \xi(T)=\xi(t_{0})\mu^{\nu}e^{-\lambda(T-t_{0})}.
\end{equation}
Using the definition of $\nu$, (\ref{xi}) leads to
\begin{equation}\label{xi1}
  \xi(T)=\xi(t_{0})\mu^{N_{0}}e^{\lambda_{0}t_{0}}e^{-(\lambda-\frac{\ln\mu}{\tau_{a}})T}.
\end{equation}
To ensure $\xi(T)\rightarrow0$ as $T\rightarrow\infty,$ it is sufficient to have $\tau_{a}>\frac{\ln\mu}{\lambda}.$ This guarantees the convergence of the impulsive differential system (\ref{im1}) to zero as time increases to infinity. Stability of (\ref{im1}) follows directly from (\ref{xi1}) and the estimate $\xi(t)\leq \xi(t_{0})\mu^{N_{0}}e^{\lambda_{0}t_{0}}$ holds if $\tau_{a}>\frac{\ln\mu}{\lambda}.$ We conclude that system (\ref{im1}) is globally asymptotically stable. Therefore, hypothesis (iv) of Theorem \ref{thm1} is also satisfied.
By Theorem \ref{thm1}, we conclude that the switched system (\ref{eq1}) is globally asymptotically stable in the $q$th mean.
\end{proof}
\section{Conclusion}
In the discussion of the stability in terms of two measures for stochastic differential functional differential, the assumption of the existence of the strong solution of the system is necessary. However, for stochastic partial differential system, this assumption is not always hold. Fortunately,
the work in \cite{1} present a method that pass on stability of strong solutions to the mild ones, and with this method, we can discuss the
the stability in terms of measures for switched stochastic partial differential delay equations with jumps. In Theorem \ref{thm1}, we presented a general frame for the stability analysis of system (\ref{eq1}). For more specific switching systems like systems under fixed-index sequence monotonicity condition and under dwell-time switching, we gave the more detail conditions that can be keep the system to be stable in terms of two measures.

\end{document}